\documentclass[11pt]{article}

\usepackage{lineno}
\usepackage{amssymb,amsthm,amsmath}
\usepackage{graphicx}
\usepackage[small]{caption}
\usepackage{subcaption}
\usepackage{epsfig}
\usepackage{amsfonts}

\usepackage[utf8]{inputenc}
\usepackage{fullpage}

\usepackage{enumerate}

\usepackage{url}
\usepackage{hyperref}

\usepackage[noline,linesnumbered]{algorithm2e}
\SetArgSty{textrm}
\let\oldnl\nl
\newcommand{\nonl}{\renewcommand{\nl}{\let\nl\oldnl}}
\makeatletter
\def\TitleOfAlgo{\@ifnextchar({\@TitleOfAlgoAndComment}{\@TitleOfAlgoNoComment}}
\def\@TitleOfAlgoAndComment(#1)#2{\nonl\hspace*{-1.5em}#2 #1\;}
\def\@TitleOfAlgoNoComment#1{\nonl\hspace*{-1.5em}#1\;}
\makeatother
\DontPrintSemicolon

\newcommand*\patchAmsMathEnvironmentForLineno[1]{
  \expandafter\let\csname old#1\expandafter\endcsname\csname #1\endcsname
  \expandafter\let\csname oldend#1\expandafter\endcsname\csname end#1\endcsname
  \renewenvironment{#1}
  {\linenomath\csname old#1\endcsname}
  {\csname oldend#1\endcsname\endlinenomath}}
  \newcommand*\patchBothAmsMathEnvironmentsForLineno[1]{
  \patchAmsMathEnvironmentForLineno{#1}
  \patchAmsMathEnvironmentForLineno{#1*}}
  \AtBeginDocument{
  \patchBothAmsMathEnvironmentsForLineno{equation}
  \patchBothAmsMathEnvironmentsForLineno{align}
  \patchBothAmsMathEnvironmentsForLineno{flalign}
  \patchBothAmsMathEnvironmentsForLineno{alignat}
  \patchBothAmsMathEnvironmentsForLineno{gather}
  \patchBothAmsMathEnvironmentsForLineno{multline}
}

\newtheorem{theorem}{Theorem}
\newtheorem{lemma}{Lemma}
\newtheorem{observation}{Observation}

\newtheorem{conjecture}{Conjecture}
\newtheorem{question}{Question}

\newcommand{\etal}{{et~al.}}
\newcommand{\ie}{{i.e.}}
\newcommand{\eg}{{e.g.}}

\newcommand{\x}{{\rm cr}}

\newcommand{\RR}{\mathbb{R}} 
\newcommand{\eps}{\varepsilon}

\def\A{\mathcal A}

\def\L{\mathcal L}
\def\O{\mathcal O}
\def\P{\mathcal P}

\newcommand{\later}[1]{}
\newcommand{\old}[1]{}

\title{\textsc{The Dirac--Goodman--Pollack Conjecture}}

\author{
Adrian Dumitrescu\thanks{%
Algoresearch L.L.C., Milwaukee, WI, USA. 
Email~\texttt{ad.dumitrescu@gmail.com}}}

\begin{document}

\maketitle

\begin{abstract}
In one of their seminal articles on allowable sequences, Goodman and Pollack gave combinatorial generalizations
for three problems in discrete geometry, one of which being the Dirac conjecture.
According to this conjecture, any set of $n$ noncollinear points in the plane has a point  incident
    to at least $c n$ connecting lines determined by the set.
The notion of allowable sequences of permutations provides a natural combinatorial setting
for analyzing these problems. Within this formalism, the conjectured generalization reads as follows:
\emph{Any nontrivial allowable $n$-sequence $\Sigma$ has  a local sequence $\Lambda_i$ whose half-period
  is at least $c n$.}
The conjecture is confirmed here with a concrete bound $c=1/845$.  Several related problems are discussed. 


\end{abstract}

\section{Introduction} \label{sec:intro}

In one of their seminal articles on allowable sequences~\cite{GP81} Goodman and Pollack gave combinatorial
generalizations (left as conjectures) for the following three problems in discrete geometry:
\begin{enumerate} \itemsep 0pt 
  \item [$\mathrm{A}$] the Erd\H{o}s--Szekeres  conjecture that any $2^{n-2}+1$ points in general position in the plane contains
    $n$ points in convex position,
  \item [$\mathrm{B}$] the Dirac conjecture that any set of $n$ noncollinear points in the plane contains  a point  incident
    to at least $c n$ connecting lines determined by the set, for some constant $c>0$, 
\item [$\mathrm{C}$] the problem of finding the minimum number of directions determined by $n$ noncollinear points in the plane.
\end{enumerate}
The notion of allowable sequences of permutations provides a natural combinatorial setting independent from geometry
for analyzing these problems and making them more transparent. Within this formalism, the conjectured generalizations
for the three statements above read as follows:
\begin{enumerate} \itemsep 0pt 
\item [$\mathrm{A'}$] \cite{GP81} Let $\Sigma$ be an allowable $2^{n-2}+1$-sequence in which only strings of length two are reversed.
  Then there are $n$ indices such that each occurs before or after the others in some term of $\Sigma$.
\item [$\mathrm{B'}$] \cite{GP81} Any nontrivial allowable $n$-sequence $\Sigma$ has  a local sequence $\Lambda_i$ whose half-period
  is at least $c n$, for some constant $c>0$. 
\item [$\mathrm{C'}$] \cite{GP81} If $\Sigma$ is a nontrivial allowable $n$-sequence, the half-period of $\Sigma$ is at least
  $2 \lfloor n/2 \rfloor$.
\end{enumerate}

The formalism as well as our results will be made precise in Section~\ref{sec:prelim},
at the end of which we state our results and review the status of the three problems and their generalizations.
It will be evident at that point that $\mathrm{A'} \implies \mathrm{A}$, $\mathrm{B'} \implies \mathrm{B}$, and
$\mathrm{C'} \implies \mathrm{C}$, though not conversely.
Our goal in this paper is the proof of Statement $\mathrm{B'}$.   

\paragraph{Connecting lines and a theorem of de Bruijn and Erd\H{o}s.}
Before discussing Dirac's conjecture, it is natural to mention how this question appeared in the broader context
of estimating the total number of connecting lines determined by a noncollinear point set
(clearly the answer is $1$ for collinear sets). Theorem~\ref{thm:erdos} below provides the answer.

Historically, the study of point sets and their connecting lines  draws from a question
asked by Sylvester~\cite{Sy893} about $50$ years earlier: \emph{For a finite set of points, not all on a line,
  does there always exist a line that contains exactly two of the points?}
If the answer is positive, the corresponding equivalent statement would read:
if for every pair of points in the set, the line determined by these points contains a third one, then
all the points are collinear. Given a point set $S$, a connecting line (\ie, a line determined by the set)
is called \emph{ordinary} if it contains precisely two points of $S$; see also~\cite{BM90}.
Sylvester problem got forgotten over time but was rediscovered by Erd\H{o}s~\cite{Erd43} in $1943$.

The first proof of existence of an ordinary line dates back to those times and it is now commonly referred to as the
\emph{Sylvester--Gallai Theorem} (solutions were found by several researchers).
Its colorful history is recounted by Chv\'atal in his recent monograph~\cite{Chv21}.
Earlier accounts on its development can be found in~\cite{BMP05,CFG91,EP95,KW91,PS09}.

\begin{theorem} \label{thm:sg} {\rm (Sylvester--Gallai)}.
Every set of $n$ noncollinear points in the plane admits an ordinary line. 
\end{theorem}

Motzkin~\cite{Mo51} was the first to show that the number of ordinary lines tends to infinity with $n$. 
Further, Kelly and Moser~\cite{KM58} proved that there are at least $3n/7$ ordinary lines,
and Csima and Sawyer [11] raised this bound to $6n/13$ for $n \geq 8$. 
Finally, Green and Tao~\cite{GT13} proved that if $n$ is sufficiently large then there are at least $n/2$ ordinary lines,
thereby settling the so called strong Dirac--Motzkin conjecture for large $n$ (even though neither of the authors seem
to have conjectured this in print, see~\cite{GT13}). On the other hand, there are arbitrarily large points sets
with no more than $n/2$  ordinary lines: for even $n$, take a regular $n/2$-gon, which determines $n/2$
directions, and the $n/2$ projective points corresponding to these directions; see~\cite[Ch.~7.2]{BMP05}.

\smallskip
Erd\H{o}s~\cite{Erd43} deduced the following interesting corollary of Theorem~\ref{thm:sg}.
Here the term \emph{near-pencil} describes a point set that is almost collinear,
in the sense that all but of one the points are collinear.

\begin{theorem} \label{thm:erdos} {\rm (Erd\H{o}s~\cite{Erd43})}.  
  For a set of $n$ noncollinear points in the plane, the number of connecting lines is always at least $n$;
  and it is equal to $n$ if and only if the points form a near-pencil.
\end{theorem}

In fact every other configuration determines more lines as quantified in the following result
of Kelly and Moser~\cite{KM58}.

\begin{theorem} \label{thm:gap} {\rm (Kelly and Moser~\cite{KM58})}. 
Let $S$ be a set of $n$ points and let $\lambda=\lambda(S)$ denote the number of connecting lines. 
If at most $n-k$ points of $S$ are collinear and $n > \frac12\{3 (3k-2)^2 + 3k-1 \}$
then $\lambda \geq kn - \frac12 (3k+2)(k-1)$. 
\end{theorem}

In particular ($k=2$), if at most $n-2$ points are collinear and $n \geq 27$, 
the number of connecting lines is always at least $2n-4$; the lower bound actually holds for $n \geq 10$,
see~\cite[Ch.~6]{KW91}.

Perhaps even more interesting than Theorem~\ref{thm:erdos} is the following result of
de Bruijn and Erd\H{o}s~\cite{BE48} from about the same time---which provides the same
answer under more general circumstances that distill the essential features present in the theorem.
See also~\cite[Ch.~19]{LW01},~\cite{Pa05}. 

\begin{theorem} \label{thm:be} {\rm (de Bruijn and Erd\H{o}s~\cite{BE48})}.
  Let $(V,E)$, $|V|=n$, $|E|=m$, be a hypergraph, where every pair of elements in $V$ is contained
  in precisely one edge in $E$. Then $m \geq n$, with equality if and only if (i)~one of the sets contains all but
  one elements of $V$ and the others are two-element sets containing the remaining element; or
  (ii)~$E$ is the system of lines of a finite projective plane defined on $V$. 
\end{theorem}

A result of Motzkin~\cite{Mo67}, Rabin, and Chakerian~\cite{Ch70} states that any set of $n$ two-colored
(say, by red or blue) noncollinear points in the plane determines a monochromatic line; see also~\cite[Ch.~13]{AZ18}. 

\paragraph{Dirac's conjecture.}
For a noncollinear set $S$ of $n$ points in the plane, let $t(S)$ be the minimum number of lines spanned be $S$
that are incident to a point in $S$; let $t(n)$ be the minimum of $t(S)$ over all point sets of size $n$.
In a dual setting, for a set $\L$ of $n$ lines in the plane, no two of which are parallel, let $r(\L)$ be the maximum
number of crossing points (vertices of the line arrangement) on a line in $\L$.

G. A. Dirac~\cite{Dir51} and T. S. Motzkin~\cite{Mo51}, independently of each other and at the same time
proposed the following problem: Does every noncollinear set of points contain some point that is incident to
at least $n/2$ lines determined by the set?
Initially Dirac proved that there are at least $\sqrt{n+1}$ lines incident to one of the points
and conjectured the existence of a point incident to at least $\lfloor n/2 \rfloor$ connecting lines.
Several counterexamples were found by Gr\"unbaum (for $n=9,15,19,25,31$, and $37$), see~\cite{Gr72}
and~\cite[F12]{CFG91}, and so the conjecture has been modified~\cite{Erd61} to read as follows:

\begin{conjecture} {\rm (Dirac).} \label{conj:dirac}
  Given a set $S$ of $n$ noncolllinear points in the plane, there exists a point in $S$ incident to
  $c n$ lines determined by $S$, for some constant $c>0$.
\end{conjecture}

From the other direction, Akiyama~\etal~\cite{AIKN11} considered the problem of finding noncollinear point sets
$S$ with $t(S) \leq \lfloor n/2 \rfloor$. They showed that for every $n \geq 8$ of the form $n=12 k + r$, $r \neq 11$, 
there exists a set $S$ of $n$ noncollinear points satisfying  $t(S) \leq \lfloor n/2 \rfloor$. 
An infinite family of counterexamples to the strong Dirac conjecture was found by Felsner~\cite[p.~313]{BMP05}.
In the dual setting it consists of $6k+7$ lines in the real projective plane (r.p.p.) where no line is incident
to more than $3k+2$ points of intersection. 
In a stronger form---the so called \emph{strong Dirac conjecture}---the bound is replaced
by $n/2 -c$, where $c>0$ is a constant~\cite{Mo75}; see also~\cite[Ch.~7.3]{BMP05}.

In $1983$ Beck~\cite{Be83} proved the following result and further observed that Conjecture~\ref{conj:dirac}
immediately follows from it. 

\begin{theorem} \label{thm:beck} {\rm (Beck~\cite{Be83}).}
  Let $S$ be a set of $n$ points in the plane. If at most $\ell$ points of $S$ are collinear,
  then $S$ determines at least $\Omega(n(n-\ell))$ distinct lines.
\end{theorem}

At about the same time Szemer\'edi and Trotter~\cite{ST83} obtained their classic result on the number
of point-lines incidences in the plane: Theorem~\ref{thm:st-83a} or~\ref{thm:st-83b} below.
Their result also implies Conjecture~\ref{conj:dirac}.
Interestingly enough, as remarked by Sz\'ekely~\cite{Sz97},
Beck obtained his result on connecting lines from a result weaker than
the Szemer\'edi--Trotter theorem. 

Here we give two equivalent formulations of the Szemer\'edi and Trotter result.
Given a point set $S$ in $\RR^2$, for any integer $k\geq 2$,
a line is called \emph{$k$-rich} if it is incident to at least $k$ points of $S$.

\begin{theorem} {\rm (Szemer\'edi--Trotter~\cite{ST83})}. \label{thm:st-83a}
The number of point-line incidences among $n$ points and $\ell$ lines in $\RR^2$ is
\[ I(n,\ell)=\O(n^{2/3}\ell^{2/3}+n+\ell). \]
\end{theorem}

\begin{theorem} {\rm (Szemer\'edi--Trotter~\cite{ST83})}.
\label{thm:st-83b}
Given $n$ points in $\RR^2$, the number of $k$-rich lines, $k\geq 2$, is
\[ \O\left(n^2/k^3+n/k\right). \]
\end{theorem}

The resulting constants, however, in the above proofs for Conjecture~\ref{conj:dirac} are quite small;
for instance, the constant obtained in~\cite{ST83} is $10^{-1087}$; see also~\cite[Ch.~6]{KW91}. 
New developments in the theory of geometric graphs have lead over time to better constants
in Dirac's conjecture. One such tool is the classic crossing lemma proved in the early 1980s by
Ajtai, Chv\'atal, Newborn, and Szemer\'edi~\cite{ACNS82} and Leighton~\cite{Le84}.
The sharper constant appearing at the end of the lemma was established recently by Ackerman~\cite{Ack19}
who improved an earlier bound by Pach, Radoi\v{c}i\'{c}, Tardos, and T\'{o}th~\cite{PRTT06};
see also~\cite[Ch.~4]{PS09}.

\begin{lemma}   \label{lem:acns}
  {\rm (Ajtai, Chv\'atal, Newborn, and Szemer\'edi~\cite{ACNS82}
Leighton~\cite{Le84}, Ackerman~\cite{Ack19}).}
  Let $G=(V,E)$, where $|V|=n$, $|E|=e$, be a simple graph with with $n$ vertices and $e \geq 4n$ edges.
  Then $\x(G) \geq c \cdot e^3/n^2$, for a suitable constant $c>0$. In particular, one may take $c=1/64$;
  and if $e \geq 6.95n$, then one may take  $c=1/29$.
\end{lemma}

Another key development is due to Sz\'ekely~\cite{Sz97}, whose groundbreaking approach of constructing suitable
graphs and running the crossing lemma machinery led to new bounds and improved constants in relation to Dirac's
conjecture, but in many other problems as well. For instance, using this approach, Payne~\cite{Pay14} showed in his
thesis that Conjecture~\ref{conj:dirac} holds with $c=2^{-15}$; see also~\cite{PW14}. 
Further, Payne and Wood~\cite{PW14} raised the bound to $c=1/37$; notably,
Hirzebruch's inequality is used in their proof; see also~\cite[Ch.~7]{BMP05}.
Pham and Phi refined the argument of Payne and Wood and improved the bound to $c=1/26$~\cite{Pha17}.  
Recently, Han~\cite{Han17} established the current best result in Dirac's conjecture,
showing that there is a point incident to $\left \lceil n/3 \right \rceil +1$ connecting lines;
notably, the Bojanowski--Pokora inequality is used in their proof.
His result answers a question of Klee and Wagon~\cite[Ch.~6]{KW91}; the same question was reposed $20$
years later by Akiyama~\etal~\cite{AIKN11}.

In regard to the proof techniques, it is worth mentioning that neither the Hirzebruch's inequality
nor the Bojanowski--Pokora inequality are known to hold in a pseudoline setting (discussed in
Section~\ref{sec:prelim}).

\paragraph{Outline of the paper.} Section~\ref{sec:prelim} gives an overview of pseudoline arrangements
and allowable sequences and lists our results. The main results are the lower bounds in Theorems~\ref{thm:main}
and~\ref{thm:extension} in relation to Conjecture~\ref{conj:dirac-g-p} in Section~\ref{sec:prelim}
(\ie, Statement $\mathrm{B'}$ at the beginning of Section~\ref{sec:intro}). 
The upper bound in Theorem~\ref{thm:deltoid} gives a partial answer to a question of
Lund, Purdy and Smith~\cite{LPS14}. 
Section~\ref{sec:main} contains the proofs of Theorems~\ref{thm:main}
and~\ref{thm:extension}. Section~\ref{sec:remarks} contains the proof of Theorem~\ref{thm:deltoid}.

\section{Pseudolines, allowable sequences, and wiring diagrams}   \label{sec:prelim}

\paragraph{Pseudoline arrangements.}      
A family (collection) of two-way infinite $x$-monotone curves in the plane is called an
\emph{(Euclidean) arrangement of pseudolines} if any two curves have precisely one point in common,
at which they properly cross~\cite[Ch.~6]{Fe04}. 
An arrangement is \emph{simple} if no three pseudolines have a common point of intersection, see 
Fig.~\ref{fig:r1r2r3}\,(left). An  arrangement is \emph{nontrivial} if not all pseudolines cross at a single point.
\begin{figure}[htpb]
\centering
\includegraphics[scale=0.47]{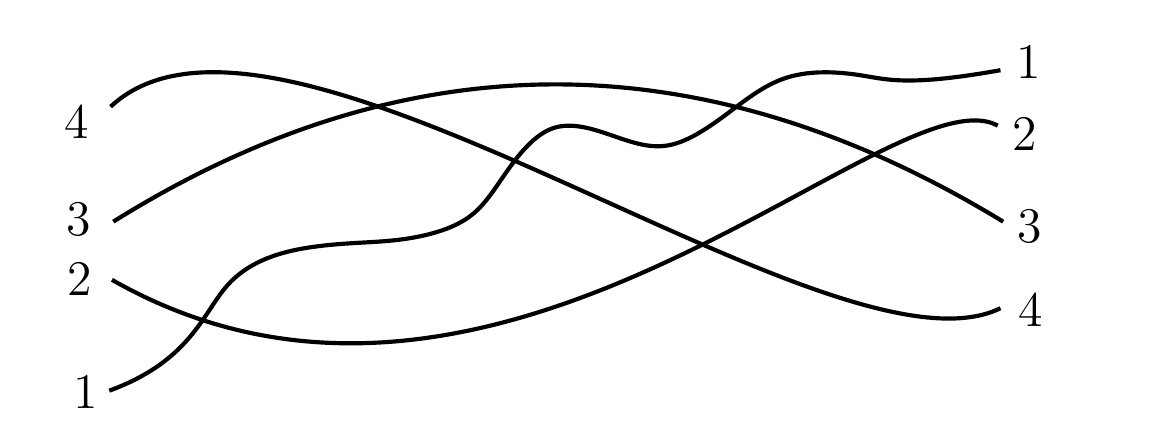}
\hspace{0.1cm}
\includegraphics[scale=0.47]{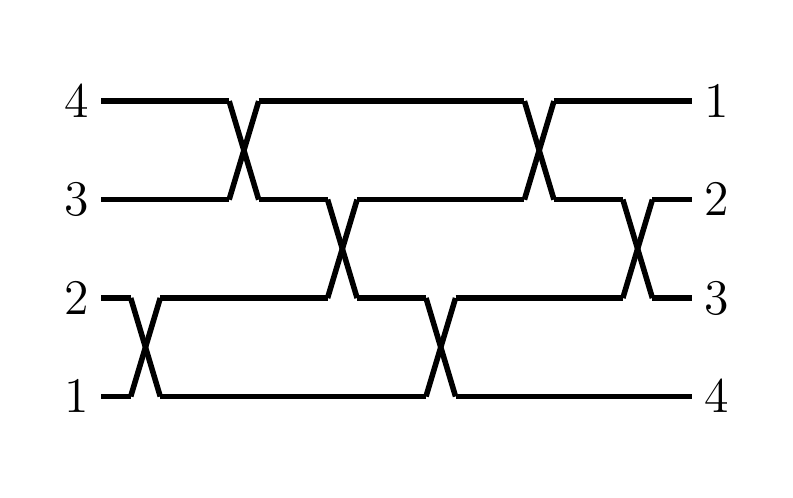}
\hspace{0.1cm}
\includegraphics[scale=0.47]{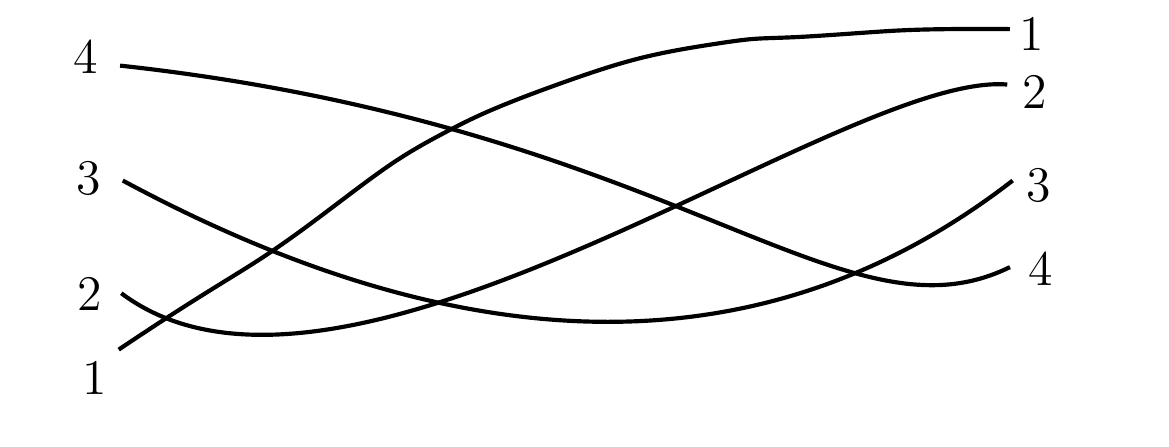}
\caption{Left: A simple arrangement $\A$. Center: Wiring 
diagram of $\A$. Right: An arrangement $\A'$ that is not 
isomorphic to the arrangement $\A$ on the left.}
\label{fig:r1r2r3}
\end{figure}

A family $\P$ of pseudolines is \emph{stretchable} if there exists a family of lines $\L$ such that the cell
decompositions induced by $\P$ and $\L$ are topologically isomorphic.
Two arrangements are \emph{isomorphic}, \ie, considered the same,
if they can be mapped onto each other by a homeomorphism of the plane~\cite{FV11};
see Fig.~\ref{fig:r1r2r3}\,(right). Equivalently, two arrangements are isomorphic if there is an isomorphism between
the induced cell decompositions~\cite[Ch.~6]{Fe04}. 
Two classic representations of pseudoline arrangements are \emph{allowable sequences}~\cite{GP80,GP93}
and \emph{wiring diagrams}~\cite{FG17}.

\paragraph{Allowable sequences.}   
Let $P$ be a set of $n$ points in the plane and assume that no two points have the same
$x$-coordinate. Label the points of $P$ by $1,2,\ldots,n$ in increasing order of their $x$-coordinate.
Take a horizontal line $\ell$ and start rotating it counterclockwise about a fixed point; 
In each position, the order of the orthogonal projections of the elements of $P$ onto $\ell$
makes a permutation of $1,2,\ldots,n$. 
As the line $\ell$ rotates counterclockwise about a fixed point, we obtain a
periodic sequence of permutations which is called the \emph{circular sequence} of the
configuration~\cite{GP81}; see also~\cite[Ch.~2]{Ed87}, \cite[Ch.~6]{Fe04}, \cite{FG17}, \cite[Ch.~1]{PS09}.
The first half-period of this sequence starts with the identity permutation $1,2,\ldots,n$
and ends with its reversal, $n,n-1,\ldots,1$; this corresponds to a rotation of $\ell$ by $180^\circ$.
During this half-period the following rule is in effect:

\begin{enumerate} \itemsep 1pt
\item Every permutation is obtained from the previous one by reversing one or more nonoverlapping
increasing subsequences of adjacent elements.
\end{enumerate}
  
If the rotation of $\ell$ continues, we obtain the same sequence of permutations as before,
except that now each of them is reversed. After a complete rotation of $360^\circ$, we get
back $1,2,\ldots,n$, the permutation we started with. And so one is usually interested only in the
sequence for the first half-period.

Goodman and Pollack~\cite{GP81} generalized this process associated with a point set to an abstract setting.
Any sequence of permutations that starts with $1,2,\ldots,n$, ends in $n,n-1,\ldots,1$, and
satisfies the above rule is called an \emph{allowable sequence} (or \emph{$n$-sequence}).
The \emph{half-period} of this sequence is one  less than the number of permutations in the sequence
\ie, the number of \emph{steps} (or \emph{moves}) in the process\footnote{If convenient, the process can be extended
  beyond the term $n,n-1,\ldots,1$, so that a periodic sequence of permutations results that cycles back to $1,2,\ldots,n$.
Terms a half-period apart are the reverses of each other. However, this extension won't be needed here.}.
An allowable sequence $\Sigma$ is \emph{simple} if any two consecutive permutations in $\Sigma$
differ by the reversal of an adjacent pair $ij$, where $i<j$.
An allowable sequence is \emph{nontrivial} if it has more than two permutations;
equivalently $1,2,\ldots,n$ is not reversed in one step. Throughout this paper,
we only consider nontrivial sequences.

One can extract an allowable sequence of permutations from any given arrangement of pseudolines by sweeping
a vertical line from left to right and recording the switches that occur in that order. 
Even though not every allowable sequence is geometrically realizable as the circular sequence generated by a
set of points (or lines), it is however true that every allowable sequence is realizable as the $n$-sequence 
generated by an arrangement of pseudolines~\cite{GP93}. Write each permutation in the sequence as a vertical column
of $n$ numbers and put the columns one after the other. The $i$th pseudoline is the piecewise linear
$x$-monotone curve obtained by connecting all occurrences of number $i$, for $i=1,2,\ldots,n$ and extended both ways
to infinity. By construction, this family of curves is a pseudoline arrangement whose sweep-sequence is the
given allowable sequence. By this equivalence, in our arguments we may use  language that applies to one setting
(allowable sequences) or the other one (pseudolines) as convenient. 

Let $\Sigma$ be an allowable $n$-sequence. For each $i \in [n]$, its \emph{local sequence} $\Lambda_i(\Sigma)$
is the sequence of reversals involving the index $i$. Obviously such reversals appear in succession, \ie,
no two are simultaneous. The \emph{half-period} (or \emph{length}) of a local sequence is the number
of reversals in the  sequence.

\begin{figure}[htpb]
\centering
\includegraphics[scale=0.47]{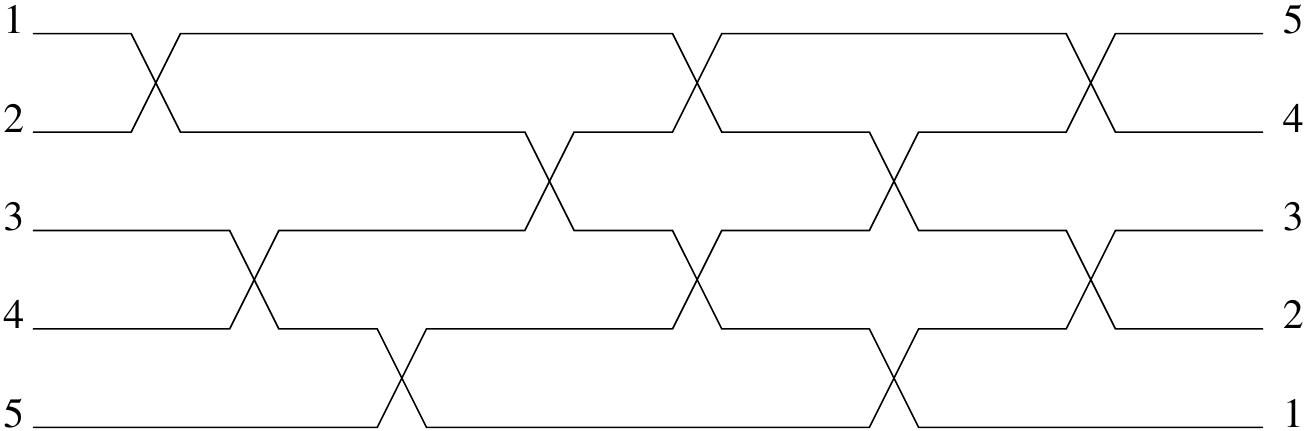}
\hspace{1cm}
\includegraphics[scale=0.47]{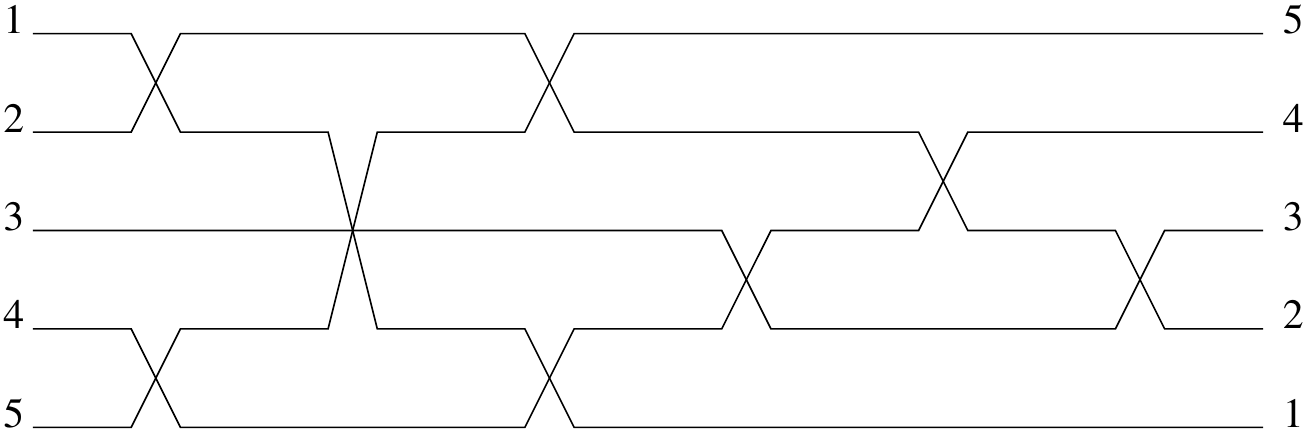}
\caption{Wiring diagrams of a simple arrangement (left) and a non-simple one (right).}
\label{fig:wiring}
\end{figure}

\paragraph{Wiring diagrams.}
A~\emph{wiring diagram} is an Euclidean arrangement of pseudolines 
consisting of piece-wise linear `wires', each horizontal except for shorter slanted segments where it crosses other wires.
Each pair of wires cross exactly once; see Fig.~\ref{fig:r1r2r3}\,(center). 
Wiring diagrams are also known as \emph{reflection networks},
\ie, networks that bring $n$ wires labeled from $1$ to $n$ into their reflection
by means of performing switches of (two or more) adjacent wires~\cite[p.~35]{Kn92}.
For example, the $5$-sequence for the wiring diagram in Fig.~\ref{fig:wiring}\,(right) is
\[ 12345 \xrightarrow{12,45} 21354\xrightarrow{135} 25314 \xrightarrow{25,14} 52341
\xrightarrow{34} 52431 \xrightarrow{24} 54231 \xrightarrow{23} 54321. \]
Its half-period is $6$. 
Its five local sequences are the following.
$\Lambda_1 = 12, 135, 14$;
$\Lambda_2 = 12, 25, 24, 23$;
$\Lambda_3 = 135, 34, 23$;
$\Lambda_4 = 45, 14, 34, 24$; and 
$\Lambda_5 = 45, 135, 25$. The half-period of $\Lambda_5$ is $3$.

\paragraph{Applications of allowable sequences.} A classic example is the result of Ungar mentioned in the introduction
on the minimum number of directions determined by $n$ noncollinear points in the plane. If $D(n)$ denotes this number,
Ungar~\cite{Un82} showed that $D(n)= 2 \lfloor n/2 \rfloor$, which is tight for the near-pencil configuration.
His proof via allowable sequences concentrates on the subsequence of switches crossing the midline
that separates the first $n/2$ elements from the last $n/2$ elements (assuming that $n$ is even). 
Another key result is one obtained by Edelsbrunner and Welzl~\cite{EW85} in the study of $k$-sets;
they showed that the number of $k$-sets in a set of $n$ points is $\O(n k^{1/2})$; see also~\cite[Ch.~2]{Ed87}. 
More recent applications can be found in~\cite{DT07} and~\cite{Ni05}. In the latter article, Nilakantan obtained an
alternative proof of Theorem~\ref{thm:sg} via allowable sequences by arguing the existence of a \emph{simple switch},
namely one that involves only two elements.

\paragraph{Dirac--Goodman--Pollack conjecture for pseudolines.}
For an arrangement $\L$ of pseudolines let $r(\L)$ be the maximum number of crossing points
(vertices of the line arrangement) on a pseudoline in $\L$.
The conjecture can be formulated in terms of allowable sequences (as mentioned in Section~ \ref{sec:intro})
or in terms of systems of pseudolines. The latter formulation is as follows. 

\begin{conjecture} {\rm \cite{Dir51,GP81}}  \label{conj:dirac-g-p}
Let $\L$ be a nontrivial arrangement of $n$ pseudolines. Then $r(\L) \geq c n$, for some constant $c>0$.
\end{conjecture}

Lund, Purdy and Smith~\cite{LPS14} claimed that such a bound holds, but did not provide any proof.
From the other direction, they constructed arrangements with $r(\L) \leq \frac49 n $.  
Regardless, here we obtain the first concrete lower bound (Theorem~\ref{thm:main})
and an extension for many pseudolines (Theorem~\ref{thm:extension}). 

\begin{theorem} \label{thm:main}
  Let $\L$ be a nontrivial arrangement of $n$ pseudolines. Then there is a pseudoline in~$\L$ that is incident
  to at least $c n$ crossing points. In particular, one may take $c=1/845$ for large $n$. 
\end{theorem}

\begin{theorem} \label{thm:extension}
  Let $0<\delta<1$ be any constant. Consider an arrangement of pseudolines in which every crossing 
  involves at most $\delta n$ elements. Then there exist $\Omega(n)$ pseudolines whose local sequences have
  length (\ie, half-period) $\Omega(n)$. 
\end{theorem}

From the other direction, an old construction studied by Rigby~\cite{Rig80} shows the following.

\begin{theorem} \label{thm:deltoid}
  There is an infinite family of arrangements of $n$ lines (as a system of pseudolines), such that
  \begin{itemize} \itemsep 0pt
  \item each vertex is incident to at most $3$ lines, and
  \item no line is incident to more than $\frac12 \, n + \O(1) $ vertices. 
  \end{itemize}
\end{theorem}

We end this section with a brief review of the status of the three problems and their generalizations (from Section~\ref{sec:intro}). 
Statement $\mathrm{C'}$ has been settled by Ungar in 1982~\cite{Un82}.
Statement $\mathrm{B'}$ is proved in Theorem~\ref{thm:main} with a bound of $n/845$ (for $n$ sufficiently large).
Statement $\mathrm{A'}$ remains open, however, recent results are closing in on this problem.
Let $f(n)$ denote the minimum number of points in general position that determine a convex $n$-gon.  
For the geometric variant $\mathrm{A}$, Erd\H{o}s and Szekeres~\cite{ES35,ES60} proved many years ago
that $2^{n-2}+1 \leq f(n) \leq 4^{n(1-o(1))}$; after several constant-factor improvements by other researchers
that we skip here, Suk~\cite{Suk17} managed to bring the upper bound to same base as the lower bound, \ie,
$f(n) \leq 2^{n+ \O(n^{2/3} \log{n})}$. Holmsen, Mojarrad, Pach, and Tardos~\cite{HMPT20} generalized Suk’s result
to pseudoline arrangements and improved the error term. The improvement carries over to the geometric variant
and implies $f(n) \leq 2^{n+ \O(\sqrt{n \log{n}})}$.

\section{Proofs of the main results}   \label{sec:main}

In this section we prove Theorems~\ref{thm:main} and~\ref{thm:extension}. 
The key component in the proof is a dual extension of the Szemer\'edi--Trotter theorem for point-line incidences
to arrangements of $x$-monotone pseudolines. We employ Sz\'ekely's method~\cite{Sz97}. 
Let $2 \leq k \leq n$. A crossing point is $k$-\emph{rich} if it is incident to at least $k$ pseudolines. 

\begin{lemma} \label{lem:dual-st}
  Let $5 \leq k \leq n-1$.
  For an arrangement of $n$ pseudolines, the number of $k$-rich crossing points is at most
  \begin{equation} \label{eq:dual-st}
    c_1 \frac{n}{k} + c_2 \frac{n^2}{k^3},
  \end{equation}
 for a suitable constants $c_1,c_2>0$.
 In particular, one may take $c_1=5$ and $c_2=125/2$; and
 if $k \geq 8$ one may take $c_1=14$ and $c_2=18.12$. 
 
\end{lemma}
\begin{proof}
  Construct a graph $G=(V,E)$ drawn in the plane, where $V$ is the set of $k$-rich crossing points in $\A$
  and edges connect vertices along the pseudolines in $\A$. Let $|V|=m$. Refer to Fig.~\ref{fig:graph}
  for an example.
\begin{figure}[htpb]
\centering
\includegraphics[scale=0.7]{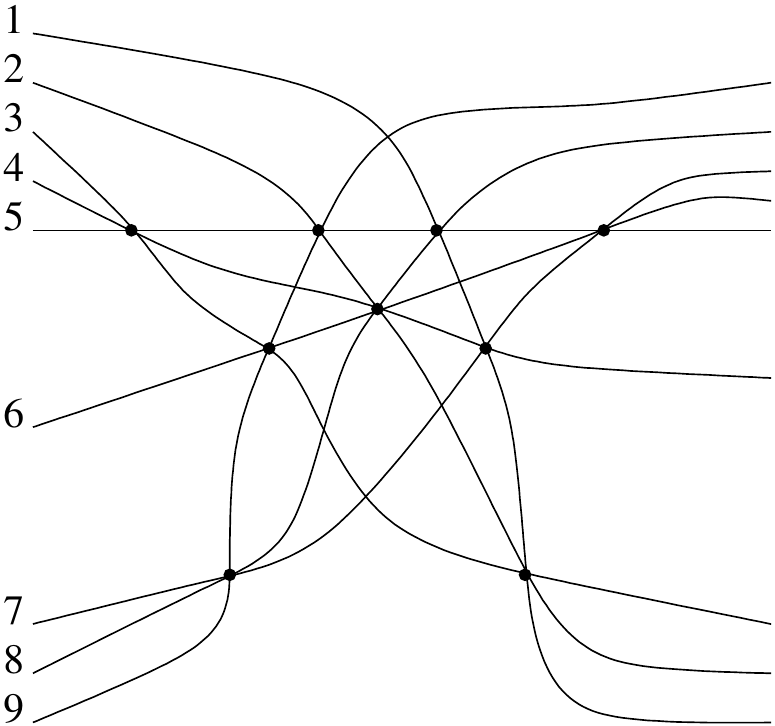}
\caption{The graph $G$; here $n=9$, $k=3$, $m=9$, and $|E|=19$.}
\label{fig:graph}
\end{figure}

  The graph $G$ is simple since every pair of pseudolines cross exactly once. 
  Since $\A$ is an arrangement of $n$ pseudolines, $\x(G) \leq {n \choose 2}$. 
  We have $|V| =m$ and $|E| \geq km -n$ by easy counting. We distinguish two cases.

  \smallskip
  \emph{Case 1.} $km \leq 5n$. Then $m \leq 5n/k$, as required.

   \smallskip
   \emph{Case 2.} $km \geq 5n$. Then $|E| \geq km - km/5 = 4km/5 \geq 4m$ by the assumption $k \geq 5$.
   The former setting of the crossing lemma (Lemma~\ref{lem:acns}) can be applied and it gives
   \[ {n \choose 2} \geq \x(G) \geq \frac{|E|^3}{64 |V|^2}, \text{ or }
   n^2 \geq \frac{2}{64} \cdot  \frac{4^3}{5^3} \cdot \frac{k^3 m^3}{m^2}. \]
   It follows that $m \leq \frac{125}{2} \, \frac{n^2}{k^3}$, as required.

   \medskip
 Assume now that $k \geq 8$. We distinguish two cases.

  \smallskip
  \emph{Case 1.} $km \leq 14n$. Then $m \leq 14n/k$, as required.

   \smallskip
   \emph{Case 2.} $km \geq 14 n$. Then $|E| \geq km - km/14 = 13 km/14 \geq 7m$ by the assumption $k \geq 8$.
     The latter setting of the crossing lemma can be applied and it gives
   \[ {n \choose 2} \geq \x(G) \geq \frac{1}{29} \frac{|E|^3}{V|^2}, \text{ or }
   n^2 \geq \frac{2}{29} \cdot  \frac{13^3}{14^3} \cdot \frac{k^3 m^3}{m^2}. \]
   It follows that $m \leq 18.12\, \frac{n^2}{k^3}$, as required.
\end{proof}

Showing that $\Sigma$ has  a local sequence $\Lambda_i$ whose half-period is $\Omega(n)$ is equivalent
to showing that at least one pseudoline is incident to $\Omega(n)$ crossing points (these may be
vertices of $G$ or edge crossings with the respective pseudoline).

\begin{observation} \label{obs:many}
  Let $\L' \subset \L$ be the subset of pseudolines participating in a crossing $\xi$ 
  and $\ell \in \L \setminus \L'$ be any other pseudoline. Then $\ell$ must cross every
  pseudoline in $\L'$ at a different crossing point.
\end{observation}
\begin{proof}
  Let $\L'' \subset \L$ be the subset of pseudolines participating in a fixed crossing other than $\xi$.
  Then $|\L' \bigcap \L''| \leq 1$, since every pair of pseudolines cross exactly once.
  Since $\ell$ must cross every other pseudoline, in particular, every pseudoline in $\L'$,
  $\ell$ must have at least $|\L'|$ different crossing points. 
\end{proof}

Note that the condition on the uniqueness of any pairwise intersection (as above) is essentially the 
same as that appearing in Theorem~\ref{thm:be}; see also~\cite[Ch.~19]{LW01}.

\paragraph{Proof of Theorem~\ref{thm:main}.}
  First assume that there exists a $k$-rich crossing point $\xi$ for $k=n/845$. Consider the
  subset of pseudolines  $\L' \subset \L$ involved in this crossing. We have $|\L'| \geq n/845$;
  recall that $|\L'| \leq n-1$.  Pick any pseudoline $\ell \in \L \setminus \L'$. By Observation~\ref{obs:many},
  $\ell$ must intersect every element in $\L'$ at a different crossing point. In other words,
  the length of $\ell$'s local sequence is at least $|\L'| \geq n/845$, as required.

We may now assume for the remainder of the proof that there are no  $k$-rich crossing points  for $k=n/845$. 
Since every pair of pseudolines intersect exactly once, the total number of pair switches is ${n \choose 2}$. 
We next compute an upper bound on the number of pair switches at the
$k$-rich crossing points  for $256 \leq k < n/845$.
Since $k \geq 8$, we can use the latter setting of Lemma~\ref{lem:dual-st},
with $c_1=14$ and $c_2=18.12$. Once this bound is obtained, 
we deduce from it (and the total count) a lower bound on the total number of switches at
$k$-rich crossing points  for $2 \leq k \leq 255$. Finally we obtain a lower bound on the
maximum number of crossings on some pseudoline in $\L$.

For $i \geq 1$, let $V_i \subset V$ denote the subset of vertices incident to at least $2^i$ and at most
$2^{i+1}-1$ pseudolines. Observe that a vertex in $V_i$ contributes fewer than
${2^{i+1} \choose 2} \leq 2 \cdot 4^i$ switches (out of ${n \choose 2}$). 

Let $N_1$ denote the number of switches at $k$-rich vertices for $256 \leq k < n/845$
contributed by the first term (linear in $n$) in Equation~\eqref{eq:dual-st}. 
Let $x$ be the minimum integer such that $2^{x+1} \geq n/845$. Then $2^x < n/845$, 
whence we have 
\begin{align*}
  N_1 &\leq c_1 n \sum_{i=8}^x \frac{1}{2^i} {2^{i+1} \choose 2} \leq 2 c_1 n \sum_{i=8}^x 2^i 
\leq 4 c_1 n \, 2^x \leq 4 c_1 n \left( \frac{n}{845}\right) \leq \frac{4.242}{64} n^2.
\end{align*}

Let $N_2$ denote the number of switches at $k$-rich vertices for $256 \leq k < n/845$
contributed by the second term (quadratic in $n$) in Equation~\eqref{eq:dual-st}. We have
\begin{align*}
  N_2 \leq  \sum_{i=8}^\infty c_2 \frac{n^2}{2^{3i}} {2^{i+1} \choose 2} 
  \leq 2 c_2 n^2 \left( \sum_{i=8}^\infty \frac{1}{2^i} \right)
  = 2 c_2 \frac{1}{128} n^2 = \frac{18.12}{64} n^2.
\end{align*}

Adding up the two contributions yields
\[ N_1 + N_2 \leq  \frac{22.362}{64} n^2. \] 
Hence at least
\begin{equation} \label{eq:rest}
  {n \choose 2} - \frac{22.362}{64} n^2 \geq \frac{9.637}{64} n^2
\end{equation}
switches occur at crossing points that involve at most $255$ pseudolines.
In the last inequality we used the fact that $n$ is large enough. 
A crossing of $j$ pseudolines, where $2 \leq j \leq 255$, distributes $j$ credits to the respective lines and uses
${j \choose 2}$ switches from the pool in~\eqref{eq:rest}. One credit received by a pseudoline counts for one
crossing point on the respective pseudoline.  The ratio of credits to switches in such a crossing, 
\[ \frac{j}{{j \choose 2}}= \frac{2}{j-1}, \] 
is minimized at $j=255$, when the ratio is $1/127$. 
Consequently, by the pigeonhole principle,  there is a pseudoline that receives at least
\[ \frac{9.637}{64} n^2 \cdot \frac{1}{127} \cdot \frac{1}{n} \geq \frac{n}{845} \]
credits, \ie, has at least this number of crossing points, as required. 
\qed

\medskip
The resemblance of the argument in the proof of Theorem~\ref{thm:main}
with the following result of Beck~\cite{Be83} is worth noting.

\begin{theorem} \label{thm:beck2} {\rm (Beck~\cite{Be83})}.
 There is  constant $c>0$ such that for any set $S$ of $n$ points in the plane, either
  \begin{enumerate} \itemsep 0pt
    \item [($\alpha$)] some line contains at least $c n$ points of $S$, or
  \item [($\beta$)] the number of distinct lines determined by $S$ is at least $c n^2$. 
\end{enumerate}
\end{theorem}

In general one cannot guarantee the existence of more pseudolines with the property in Theorem~\ref{thm:main}.
Indeed, consider the $n$-sequence of permutations:
\[ 1,2,\ldots,n-1,n \xrightarrow{1,2,\ldots,n-1} n-1,n-2,\ldots,2,1,n \xrightarrow{1,n} n-1,n-2,\ldots,2,n,1
\xrightarrow{} \cdots \xrightarrow{} n,n-1,\ldots,2,1. \]
The only pseudoline whose local sequence is of length $\Omega(n)$ is the $n$th one.
However, under very mild conditions, the stronger statement in Theorem~\ref{thm:extension} is in effect.
Its proof is analogous to that of Theorem~\ref{thm:main}.    

\paragraph{Proof of Theorem~\ref{thm:extension} (sketch).}
  Assume first that there exists a $k$-rich crossing point $\xi$ for $k=c n$, for some positive constant
  $c \leq \delta$. Consider the subset of pseudolines  $\L' \subset \L$ involved in this crossing. We have
  $c n \leq |\L'| \leq \delta n$. By Observation~\ref{obs:many}, for every pseudoline $\ell \in \L \setminus \L'$, 
  the length of $\ell$'s local sequence is at least $c n$. Moreover, $|\L \setminus \L'| \geq (1-\delta) n$,
  as required. 

  If there is no $k$-rich crossing point for $k= c n$, for a sufficiently small $c>0$, the proof is finished as before,
  by obtaining an $\Omega(n^2)$ lower bound analogous to~\eqref{eq:rest}. We omit the details.
\qed

\smallskip
It should be noted that Theorem~\ref{thm:extension} is a dual extension of Beck's result mentioned above.

\section{Upper bound questions and concluding remarks}   \label{sec:remarks}

In this section we prove Theorem~\ref{thm:deltoid}. 
In $2014$, Lund, Purdy, and Smith~\cite{LPS14} demonstrated an infinite family of (nontrivial) pseudoline
arrangements, in which an arrangement of $n$ pseudolines has no member incident to more than $4n/9$
points of intersection (\ie, vertices of the arrangement), and thereby showed that the strong Dirac conjecture
does not hold for pseudolines. 
One feature of the respective family of arrangements is that they contain vertices with high incidence,
in particular, about $n/3$ pseudolines are incident to a single vertex. The authors asked the following.

\begin{question} {\rm (Lund, Purdy, and Smith~\cite{LPS14})}   \label{question:lps}
  Is there an infinite family of arrangements of $n$ pseudolines, such that
  \begin{itemize} \itemsep 0pt
  \item no vertex is incident to $\Omega(n)$ pseudolines, and
  \item no pseudoline is incident to more than $(1-\eps) \, \frac{n}{2}$ vertices, for some
    constant $\eps>0$?
\end{itemize}
\end{question}

The authors further relaxed the second requirement by replacing $(1-\eps) \, \frac{n}{2}$ with $c n$, where $c<1$
is a constant, and asked for such an arrangement. Here we give a positive answer to the latter question,
while we show that the first requirement can be substantially strengthened in that case. 
Interestingly enough, the best construction we found uses straight lines as we were not able to exploit
the power of curved pseudolines. The features of the construction are described in Theorem~\ref{thm:deltoid}.
It is worth noting, however, that this construction falls short of answering Question~\ref{question:lps}.

\paragraph{The deltoid construction.}
The construction can be traced back to Rigby~\cite{Rig80} who provided an analysis, and even further back;
for instance,  an illustration can be found in~\cite[Ch.~8]{Loc61}.
This line arrangement has been also used in~\cite{BP18,FP84}, where descriptions and useful properties can be found. 
Let $\ell(\theta)$ denote the line connecting $e^{i\theta}$ and $e^{i(\pi -2\theta)}$ on the unit circle, with the
understanding that $\ell(\theta)$ is the tangent line when the two points coincide.
We take the freedom to denote the construction in this way based on the fact that $\ell(\theta)$ 
envelops a \emph{deltoid} as $\theta$ varies. Its key property stems from the following.

\begin{lemma} {\rm (Rigby~\cite{Rig80}, F{\"u}redi and Pal{\'a}sti~\cite{FP84}).}  \label{lem:deltoid}
  The lines $\ell(\alpha)$,  $\ell(\beta)$,  and $\ell(\gamma)$ are concurrent if and only if
  $\alpha + \beta + \gamma \equiv 0 \pmod {2\pi} $. 
\end{lemma}

\begin{figure}[htpb]
\centering
\includegraphics[scale=0.55]{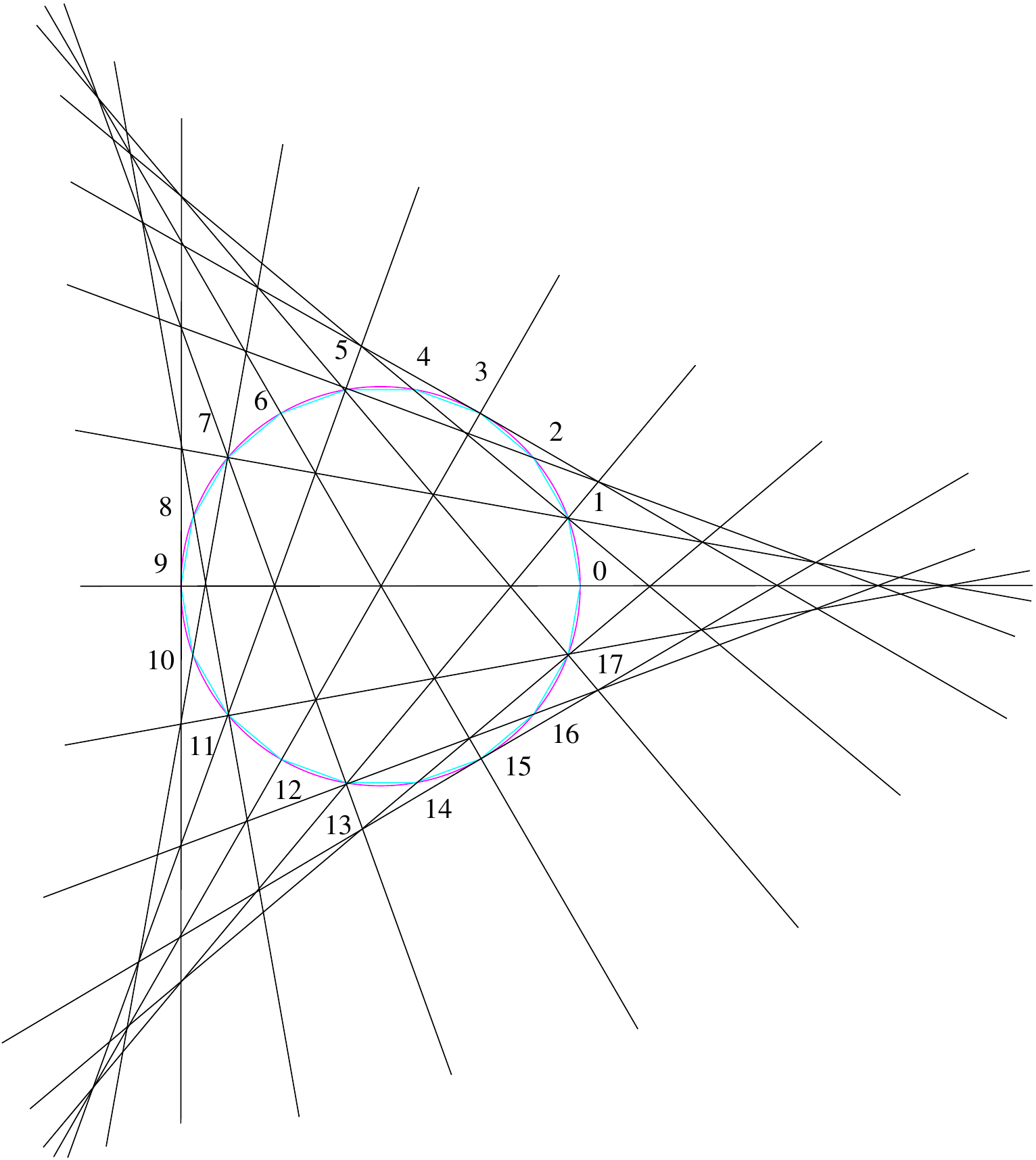}
\caption{The deltoid construction for $n=18$ lines.}
\label{fig:deltoid}
\end{figure}

Consider a regular $n$-gon inscribed in the unit circle centered at the origin, where $n$ is even,
and refer to Fig.~\ref{fig:deltoid}. 
Let $p_0,p_1,\ldots, p_{n-1}$ denote its vertices labeled counterclockwise starting from $p_0=(1,0)$.
For $i=0,1,\ldots,n-1$, draw the lines connecting $p_i$ with $p_{n/2-2i}$, where indices are considered modulo $n$.
If the points $p_i$ and $p_{n/2-2i}$ coincide, draw the tangent line to the circle at $p_i$. The resulting arrangement
has $n$ lines, $1 + \left \lceil \frac{n(n-3)}{6} \right \rceil$ triple points (\ie, vertices incident to
$3$ lines), and $n-3 +\delta(n)$ double points (\ie, ordinary vertices), where
$\delta(n)=0$ if $n \equiv 0 \pmod 3$ and $\delta(n)=2$ otherwise; see, \eg, \cite{BP18}.  
Moreover, double points may only appear on the outer envelope, whence each line is incident to at most $3$
double points, and Theorem~\ref{thm:deltoid} follows.
Obviously, the constant $1/2$ in the theorem is the best possible (for lines or pseudolines) under the first constraint.

\medskip
Determining the right constant in Conjecture~\ref{conj:dirac-g-p} remains an interesting open problem.
It is easy to obtain small improvements in the lower bound by slightly adjusting the parameters in the
proof of Theorem~\ref{thm:main}. 
Since we suspect that the answer is much closer to the best known upper bound of Lund, Purdy and Smith,
we did not insist in that direction.


\begin{thebibliography}{99}
\itemsep 3pt

\bibitem{Ack19}
Eyal Ackerman,
On topological graphs with at most four crossings per edge,
\emph{Computational Geometry: Theory and Applications}
\textbf{85} (2019), article 101574.

\bibitem{AZ18}
Martin Aigner and G{\"u}nter M. Ziegler, 
\emph{Proofs from the Book},
6th~edition, Springer, Berlin, 2018.

\bibitem {ACNS82}
Mikl{\'o}s Ajtai, Va\v{s}ek Chv\'atal, Monroe M.~Newborn, and Endre~Szemer\'edi,
Crossing-free subgraphs,
\emph{Annals of Discrete Mathematics}
\textbf{12} (1982), 9--12.

\bibitem {AIKN11}
Jin Akiyama, Hiro Ito, Midori Kobayashi, and Gisaku Nakamura,
Arrangements of $n$ points whose incident-line-numbers are at most $n/2$,
\emph{Graphs and Combinatorics}
\textbf{27(3)} (2011), 321--326.
    
\bibitem{Be83}
J{\'o}zsef Beck,
On the lattice property of the plane and some problems of
Dirac, Motzkin and Erd\H{o}s in combinatorial geometry,
\emph{Combinatorica}
\textbf{3} (1983), 281--297.

\bibitem{BP18}
J{\"u}rgen Bokowski and Piotr Pokora,
On the Sylvester--Gallai and the orchard problem for pseudoline arrangements,
\emph{Periodica Mathematica Hungarica}
\textbf{77(2)} (2018), 164--174.

\bibitem{BM90}
Peter Borwein and William O. J. Moser,
A survey of Sylvester's problem and its generalizations,
\emph{Aequationes Mathematicae}
\textbf{40(1)} (1990), 111--135.

\bibitem{BMP05}
Peter Bra\ss , William Moser, and J\'anos Pach,
\emph{Research Problems in Discrete Geometry},
Springer, New York, 2005.

\bibitem{Ch70}
Don Chakerian,
Sylvester's problem on collinear points and a relative,
\emph{The American Mathematical Monthly}
\textbf{77(2)} (1970), 164--167.

\bibitem{Chv21}
Va\v{s}ek Chv\'atal,
\emph{The Discrete Mathematical Charms of Paul Erd\H{o}s},
Cambridge University Press, New York, 2021.

\bibitem{CFG91}
Hallard T.~Croft, Kenneth~J.~Falconer, and Richard~K.~Guy,
  \emph{Unsolved Problems in Geometry},
  Springer, New York, 1991.

\bibitem{Dir51}
Gabriel A.~Dirac,
Collinearity properties of sets of points,
\emph{The Quarterly Journal of Mathematics}
\textbf{2(1)} (1951), 221--227.

\bibitem{DT07}
Adrian Dumitrescu and Csaba D. T\'oth,
Distinct triangle areas in a planar point set,
\emph{Proc. 12th Conference on Integer Programming and Combinatorial Optimization}
(IPCO 2007), vol. 4513 of LNCS, Springer, pp. 119--129.

\bibitem{Ed87}
  Herbert Edelsbrunner,
  \emph{Algorithms in Combinatorial Geometry},
 Springer, Berlin, 1987.

\bibitem{EW85}
Herbert Edelsbrunner and Emo Welzl,
On the number of line separations of a finite set in the plane,
\emph{Journal of Combinatorial Theory, Series A}
\textbf{38(1)} (1985), 15--29.

\bibitem{Erd43}
Paul Erd\H{o}s,
Three point collinearity,
\emph{American Mathematical Monthly}
\textbf{50} (1943), 65.
 
\bibitem{Erd61}
Paul Erd\H{o}s,
Some unsolved problems,
\emph{Publ. Math. Inst. Hungar. Acad. Sci.}
\textbf{6} (1961), 221--254.

\bibitem{BE48}
Nicolaas G. de Bruijn and Paul Erd\H{o}s,
On a combinatorial problem,
\emph{Proc. Kon. Ned. Akad. v. Wetensch.} 
\textbf{51} (1948), 1277--1279.
 
\bibitem{EP95}
Paul Erd\H{o}s and George Purdy,
Extremal problems in combinatorial geometry,
in \emph{Handbook of Combinatorics} (vol. 1),
(Ronald Graham, L{\'a}szl{\'o} Lov{\'a}sz, and Martin Gr{\"o}tschel, editors),
Elsevier, 1995, pp.~809--874.

\bibitem{ES35}
  Paul Erd\H{o}s and Gy\"orgy Szekeres,
  A combinatorial problem in geometry,
\emph{Compositio Mathematica} \textbf{2} (1935), 463--470.

\bibitem{ES60}
Paul Erd\H{o}s and Gy\"orgy Szekeres,
On some extremum problems in elementary geometry,
\emph{Annales Universitatis Scientiarium Budapestinensis de Rolando E\"otv\"os Nominatae Sectio Mathematica}
\textbf{3-4} (1960), 53--62.

\bibitem{Fe04}
Stefan Felsner,
\emph{Geometric Graphs and Arrangements},
Advanced Lectures in Mathematics, Vieweg Verlag, 2004.

\bibitem{FG17}
Stefan Felsner and Jacob E.~Goodman,
Pseudoline arrangements,
in \emph{Handbook of Discrete and Computational Geometry} (3rd edition), 
(J.~E.~Goodman, J.~O'Rourke, C.~D.~T\'oth, editors),
CRC Press, Boca Raton, 2017, pp.~125--157. 

\bibitem{FV11}
Stefan Felsner and Pavel Valtr,
Coding and counting arrangements of pseudolines,
\emph{Discrete {\&} Computational Geometry}
\textbf{46(4)} (2011), 405--416.

\bibitem{FP84}
Zolt{\'a}n F{\"u}redi and Ilona Pal{\'a}sti,
Arrangements of lines with a large number of triangles,
\emph{Proceedings of the American Mathematical Society}
\textbf{92(4)} (1984), 561--566.

\bibitem{GP80}
Jacob E.~Goodman and Richard Pollack,
On the combinatorial classification of nondegenerate configurations in the plane,
\emph{Journal of Combinatorial Theory Ser. A}
\textbf{29} (1980), 220--235.

\bibitem{GP81}
Jacob E.~Goodman and Richard Pollack,
A combinatorial perspective on some problems in geometry,
\emph{Congressus Numerantium}
\textbf{32} (1981), 383--394.

\bibitem{GP93}
Jacob E.~Goodman and Richard Pollack,
Allowable sequences and order types in discrete and computational geometry.
in \emph{New Trends in Discrete and Computational Geometry} (J\'anos Pach, editor), 
Algorithms and Combinatorics, Volume 10, Springer, New York, 1993, pp.~103--134.

\bibitem{GT13}
Ben Green and Terence Tao,
On sets defining few ordinary lines,
\emph{Discrete \& Computational Geometry}
\textbf{50(2)} (2013), 409--468.

\bibitem{Gr72}
Branko Gr\"{u}nbaum,
\emph{Arrangements and Spreads},
Amer. Math. Soc., Providence, 1972.  

\bibitem{Han17}
 Zeye Han,
A Note on the weak Dirac conjecture,
\emph{Electron. J. Comb.}
\textbf{24(1)} (2017), \#P1.63.

\bibitem{HMPT20}
  Andreas Holmsen, Hossein N. Mojarrad, J{\'a}nos Pach, and G{\'a}bor Tardos,
Two extensions of the Erd{\H{o}}s--Szekeres problem,
\emph{Journal of the European Mathematical Society}
\textbf{22(12)} (2020), 3981--3995.

\bibitem{KM58}
Leroy M. Kelly and William O. J. Moser,
On the number of ordinary lines determined by $n$ points,
\emph{Canadian Journal of Mathematics}
\textbf{10} (1958), 210--219.

\bibitem{KW91}
Victor Klee and Stan Wagon,
\emph{Old and New Unsolved Problems in Plane Geometry and Number Theory},  
Mathematical Association of America, Washington, DC, 1991.

\bibitem{Kn92}
Donald E.~Knuth,
\emph{Axioms and Hulls},
Lecture Notes in Computer Science, Vol.~606, Springer, Berlin, 1992.

\bibitem{Le84}
Thomas Leighton,
New lower bound techniques for VLSI,
\emph{Math. Systems Theory}
\textbf{17} (1984), 47--70.

\bibitem{LW01}
Jacobus H. van Lint and Richard M. Wilson,
\emph{A Course in Combinatorics},
Cambridge University Press, 2nd edition, New York, 2001.

\bibitem{Loc61}
Edward H. Lockwood,
\emph{A Book of Curves},
Cambridge University Press, 1961.

\bibitem{LPS14}
Ben Lund, George B. Purdy, and Justin W. Smith,
A Pseudoline counterexample to the strong Dirac conjecture,
\emph{Electron. J. Comb.}
\textbf{21(2)} (2014), \#P2.31.

\bibitem{Mo51}
Theodore S. Motzkin,
The lines and planes connecting the points of a finite set,
\emph{Transactions of the American Mathematical Society}
\textbf{70(3)} (1951), 451--464.

\bibitem{Mo67}
Theodore S. Motzkin,
Nonmixed connecting lines. Abstract 67T 605,
\emph{Notices Amer. Math. Soc}
\textbf{14} (1967), 837.

\bibitem{Mo75}
Theodore S. Motzkin,
Sets for which no point lies on many connecting lines,
\emph{Journal of Combinatorial Theory, Series A}
\textbf{18(3)} (1975), 345--348.

\bibitem{Ni05}
Niranjan Nilakantan,
Extremal problems related to the Sylvester–Gallai theorem,
in \emph{Combinatorial and Computational Geometry}
(Jacob E. Goodman, J\'anos Pach, Emo Welzl, editors), 
MSRI Publications, Volume 52, 2005, pp.~479--494.

\bibitem {Pa05}
J\'anos Pach,
Directions in combinatorial geometry,
\emph{Jahresbericht der Deutschen Mathematiker-Vereinigung}
\textbf{107} (2005), 215--225.

\bibitem {PRTT06} 
J\'{a}nos Pach, Rado\v{s} Radoi\v{c}i\'{c}, G\'{a}bor Tardos, and G\'{e}za T\'{o}th,
Improving the crossing lemma by finding more crossings in sparse graphs,
\emph{Discrete {\&} Computational Geometry}
\textbf{36(4)} (2006), 527--552.

\bibitem {PS09} 
J\'anos Pach and Micha Sharir,
\emph{Combinatorial Geometry and Its Algorithmic Applications---The Alcal\'{a} Lectures},
Mathematical Surveys and Monographs, Vol. 152, American Mathematical Society,
Providence, RI, 2009.

\bibitem{Pay14}
Michael S. Payne,
\emph{Combinatorial Geometry of Point Sets with Collinearities},
PhD thesis, The University of Melbourne, Dept. of Mathematics and Statistics, 2014. 

\bibitem{PW14}
Michael S. Payne and David R. Wood,
Progress on Dirac's conjecture,
\emph{Electron. J. Comb.}
\textbf{21(2)} (2014), \#P2.12.

\bibitem{Pha17}
 Hoang Ha Pham and Tien Cuong Phi, 
 A new progress on weak Dirac conjecture,
 preprint, 2016, \url{arXiv:1607.08398}.

\bibitem{Rig80}
John F. Rigby,
Multiple intersections of diagonals of regular polygons, and related topics,
\emph{Geometriae Dedicata}
\textbf{9(2)} (1980), 207--238.

\bibitem{Suk17}
Andrew Suk,
On the {E}rd{\H o}s--Szekeres convex polygon problem,
\emph{Journal of the American Mathematical Society}
\textbf{30} (2017), 1047--1053.

\bibitem{Sy893}
James Joseph Sylvester,
Mathematical question 11851,
\emph{Educational Times}
\textbf{46} (1893), 156.

\bibitem{Sz97} L\'aszl\'o Sz\'ekely,
Crossing numbers and hard {Erd\H{o}s} problems in discrete geometry,
\emph{Combinatorics, Probability and Computing}
\textbf{6} (1997), 353--358.

\bibitem{ST83}
Endre Szemer\'edi and William T.~Trotter,
Extremal problems in discrete geometry,
\emph{Combinatorica} \textbf{3} (1983), 381--392.

\bibitem{Un82}
Peter Ungar, $2N$ noncollinear points determine at least $2N$ directions,
\emph{Journal of Combinatorial Theory Ser. A}
\textbf{33} (1982), 343--347.

\end{thebibliography}
\end{document}